\documentclass[12pt]{amsart}
\pdfoutput=1
\usepackage[margin=1in]{geometry}
\usepackage[utf8]{inputenc}
\usepackage{graphicx, cite}
\usepackage{amsmath,amssymb}
\usepackage{amsthm}
\usepackage{color}
\usepackage{tikz}
\usepackage[utf8]{inputenc}
\usepackage{pgfplots}
\pgfplotsset{compat=1.18}
\usetikzlibrary{fillbetween}
\allowdisplaybreaks

\makeatletter
\def\l@section{\@tocline{1}{12pt plus2pt}{0pt}{}{\bfseries}}
\def\l@subsection{\@tocline{2}{0pt}{2pc}{2pc}{}}
\makeatother
\makeatletter
\def\subsection{\@startsection{subsection}{2}{\z@}%
	{-3.25ex\@plus -1ex \@minus -.2ex}%
	{1.5ex \@plus .2ex}%
	{\normalfont\bfseries\boldmath}}
\def\subsubsection{\@startsection{subsubsection}{3}%
	\z@{.5\linespacing\@plus.7\linespacing}{-.5em}%
	{\normalfont\bfseries\boldmath}}
\renewcommand\paragraph{\@startsection{paragraph}{4}{\z@}%
	{3.25ex \@plus1ex \@minus.2ex}%
	{-1em}%
	{\normalfont\normalsize\bfseries}}
\makeatother

\theoremstyle{plain}
\newtheorem{thm}{Theorem}[section]

\newtheorem{lem}[thm]{Lemma}
\newtheorem{prop}[thm]{Proposition}

\theoremstyle{definition}

\theoremstyle{remark}
\newtheorem{rem}[thm]{Remark}

\theoremstyle{plain}

\numberwithin{equation}{section}

\theoremstyle{plain} % just in case the style had changed
\newcommand{\thistheoremname}{}
\newtheorem{genericthm}[thm]{\thistheoremname}

  \newtheorem*{genericthm*}{\thistheoremname}
\newenvironment{namedthm*}[1]
  {\renewcommand{\thistheoremname}{#1}%
   \begin{genericthm*}}
  {\end{genericthm*}}

\newcommand{\R}{{\mathbb R}}
\newcommand{\SSS}{{\mathbb S}}

\newcommand{\Z}{{\mathbb Z}}

\newcommand{\im}{{\rm Im}\,}
\newcommand{\re}{{\rm Re}\,}

\newcommand{\supp}{{\textnormal{supp}}}

\makeatletter
\newcommand{\vast}{\bBigg@{4}}
\newcommand{\Vast}{\bBigg@{5}}
\makeatother

\def\udot#1{\ifmmode\oalign{$#1$\crcr\hidewidth.\hidewidth
    }\else\oalign{#1\crcr\hidewidth.\hidewidth}\fi}

\def\R{\mathbb{R}}
\def\Z{\mathbb{Z}}

\def\beq{\begin{equation}}
\def\eeq{\end{equation}}
\makeatletter
\newcommand{\doublewidetilde}[1]{{%
  \mathpalette\double@widetilde{#1}%
}}
\newcommand{\double@widetilde}[2]{%
  \sbox\z@{$\m@th#1\widetilde{#2}$}%
  \ht\z@=.9\ht\z@
  \widetilde{\box\z@}%
}
\makeatother

\def\one{\mbox{1\hspace{-4.25pt}\fontsize{12}{14.4}\selectfont\textrm{1}}}

\usepackage{tikz}

\makeatletter
\def\@makefnmark{%
  \leavevmode
  \raise.9ex\hbox{\fontsize\sf@size\z@\normalfont\tiny\@thefnmark}}
\makeatother

\begin{document}

\title[]{Fourier decay and endpoint estimates for lacunary maximal operators}

\author{Xiaojing Zhou}
\address{(Xiaojing Zhou) Department of Mathematics and Statistics\\
         Auburn University\\
         Auburn, Alabama, U.S.A, 36849}
\email{xiz0003@auburn.edu}

\begin{abstract}
We study endpoint estimates for lacunary maximal operators
associated with compactly supported measures. For every $d\geq1$,
we construct a compactly supported, absolutely continuous probability
measure on $\R^d$ with polynomial Fourier decay whose associated
lacunary maximal operator fails to map $H^1(\R^d)$ boundedly into
$L^{1,\infty}(\R^d)$. In particular, it fails to be of weak type $(1,1)$.

This gives a negative answer to the weak type $(1,1)$ question posed
by Seeger and Wright in 2011 and to Open Problems 14 and 15,
contributed by Wright at the 2016 Conference in Harmonic
Analysis, held at the University of
Wisconsin--Madison in honor of Professor Michael Christ.
\end{abstract}

\date{October 14, 2026}
\maketitle

\section{Introduction}

Let $\mu$ be a compactly supported positive Borel measure on $\R^d$.
For each $k\in\Z$, define the dilated measure $\mu_k$ by
$$
 \int_{\R^d}f(y)\,d\mu_k(y)
 :=\int_{\R^d}f(2^k y)\,d\mu(y)
$$
for every bounded Borel function $f$ on $\R^d$. The corresponding
\emph{lacunary maximal operator} is defined as
$$
 M_\mu f(x)
 :=\sup_{k\in\Z}|f*\mu_k(x)|
 =\sup_{k\in\Z}\left|\int_{\R^d}f(x-2^k y)\,d\mu(y)\right|.
$$
The study of the lacunary maximal operator is a \emph{central} topic in harmonic
analysis. Classical contributions include Calder\'on's work on
lacunary spherical means \cite{Calderon1979} and Christ's endpoint
estimates \cite{Christ1988}. We refer to Stein's monograph
\cite{Stein1993} for background on maximal averages, and to Seeger
and Wright \cite{SeegerWright2011} for results and open problems
concerning more general measures.

The \emph{goal} of this paper is to study weak type $(1,1)$ estimates for
$M_\mu$. More precisely, we ask whether there exists a finite constant
$A_\mu$ such that
\begin{equation}\label{eq:weak}
 |\{x\in\R^d:M_\mu f(x)>\lambda\}|
 \leq\frac{A_\mu}{\lambda}\|f\|_{L^1(\R^d)}
\end{equation}
for every $f\in L^1(\R^d)$ and $\lambda>0$. Here and below, $|E|$
denotes the Lebesgue measure of a measurable set $E$, and we use the
Fourier transform
$$
 \widehat\mu(\xi):=\int_{\R^d}e^{-2\pi i x\cdot\xi}\,d\mu(x).
$$

We also consider the corresponding estimate for the real Hardy space
$H^1(\R^d)$:
\begin{equation}\label{eq:hardy-weak}
 |\{x\in\R^d:M_\mu f(x)>\lambda\}|
 \leq\frac{B_\mu}{\lambda}\|f\|_{H^1(\R^d)},
 \qquad f\in H^1(\R^d),\quad\lambda>0.
\end{equation}
Here we use the atomic norm on $H^1(\R^d)$. An $H^1$ atom is a
function $a$ supported in a cube $Q$ such that $\int a=0$ and
$\|a\|_{L^\infty(\R^d)}\leq |Q|^{-1}$, and $\|f\|_{H^1(\R^d)}$ is the infimum
of $\sum_j|c_j|$ over atomic representations $f=\sum_jc_ja_j$.
In particular, $\|a\|_{H^1(\R^d)}\leq1$ for every atom $a$, and
$\|f\|_{L^1(\R^d)}\leq\|f\|_{H^1(\R^d)}$. Thus failure of \eqref{eq:hardy-weak}
implies failure of \eqref{eq:weak}.

Our main result is the following.

\begin{thm}\label{thm:main}
For every integer $d\geq1$, there exists a compactly supported,
absolutely continuous probability measure $\mu$ on $\R^d$ such that
\begin{equation}\label{eq:main-decay}
 |\widehat\mu(\xi)|\leq C_d(1+|\xi|)^{-1/8},
 \qquad \xi\in\R^d,
\end{equation}
but $M_\mu$ does not map $H^1(\R^d)$ boundedly into
$L^{1,\infty}(\R^d)$. In particular, $M_\mu$ is not of weak type $(1,1)$.
\end{thm}

It is well known that if a compactly supported finite measure $\mu$
satisfies
\begin{equation}\label{eq:fourier-hyp}
 |\widehat\mu(\xi)|\leq C(1+|\xi|)^{-\varepsilon}
 \qquad\text{for some }\varepsilon>0,
\end{equation}
then $M_\mu$ is bounded on $L^p(\R^d)$ for every $1<p<\infty$;
see \cite{SeegerWright2011}. At the endpoint, several positive
results are known under additional assumptions on the underlying
measure. In the spherical case, Christ \cite{Christ1988} proved
$H^1(\R^d)\to L^{1,\infty}(\R^d)$ boundedness, while Seeger,
Tao, and Wright \cite{SeegerTaoWright2003} obtained a weak
$L\log\log L$ estimate, subsequently refined by Cladek and Krause
\cite{CladekKrause2024}. For more general measures, Seeger and
Wright \cite{SeegerWright2011} established
$H^1(\R^d)\to L^{1,\infty}(\R^d)$ boundedness under a dimensional
assumption on the support together with a corresponding
$L^p$ smoothing assumption.

The question addressed here is whether the Fourier decay condition
\eqref{eq:fourier-hyp} alone suffices at the endpoint. More precisely,
Seeger and Wright \cite{SeegerWright2011} asked whether
\eqref{eq:fourier-hyp} implies the weak type $(1,1)$ estimate
\eqref{eq:weak}. This question also appears as Problem~14,
contributed by Wright, in
\cite{GrafakosOliveiraPramanikSeegerStovall2017}. Problem~15 in
the same collection asks whether \eqref{eq:fourier-hyp} implies
the Hardy space estimate \eqref{eq:hardy-weak}.
To the best of our knowledge, Theorem~\ref{thm:main} provides
a negative answer to both questions as stated for general
compactly supported finite measures.

\begin{rem}
The exponent $1/8$ in \eqref{eq:main-decay} is obtained from
our construction. It would be interesting to determine the
optimal range of Fourier decay exponents for which the
endpoint estimate \eqref{eq:hardy-weak} can still fail.
We do not claim that $1/8$ is the best possible exponent
for such counterexamples.
\end{rem}

The \emph{key} ingredient of the proof is the construction of a
family of probability measures $\mu_N$ on $\R$ whose Fourier
decay estimates are uniform in $N$
(Lemma~\ref{lem:uniform-decay}), while the best constants in
the $H^1\to L^{1,\infty}$ estimates for $M_{\mu_N}$ tend to
infinity along $N=2^L$ (Proposition~\ref{prop:lower}).
We then form a weighted sum of these measures, choosing the
weights and scales so that the Fourier decay is retained and
cancellation does not destroy the lower bounds.

The rest of the paper is organized as follows.
In Section~\ref{sec:blocks}, we construct the measures $\mu_N$
using Lemma~\ref{lem:sampling} and establish their uniform
Fourier decay. In Section~\ref{sec:fixed}, we prove the lower
bounds using $H^1$ atoms and construct a single measure for
which \eqref{eq:hardy-weak} fails. A product construction
then gives the result in higher dimensions.
Finally, in Remark~\ref{rem:singular}, we adapt this product
construction to obtain singular measures with the same
Fourier decay and failure of \eqref{eq:hardy-weak}
in every dimension $d\geq3$.

Throughout this paper, for nonnegative quantities $a$ and $b$,
we write $a\lesssim b$ if $a\leq Cb$ for some constant $C>0$
independent of $a$ and $b$. We write $a\simeq b$ if both
$a\lesssim b$ and $b\lesssim a$ hold. The symbol $\one_E$
denotes the characteristic function of a set $E$. Constants denoted by $C$ may change from line to line, and $C_d$ denotes a constant depending only on $d$. Finally, we use $\mathbb E$ and $\mathbb P$ to denote expectation
and probability, respectively.

\section{Construction of the measures}\label{sec:blocks}

\subsection{Choice of the points}

The following lemma is important.

\begin{lem}\label{lem:sampling}
There is an integer $N_0\geq2$ such that, for every integer
$N\geq N_0$, there exist points
$$
 x_j\in I_j
 :=\left[1+\frac{j-1}{N},\,1+\frac jN\right],
 \qquad 1\leq j\leq N,
$$
satisfying
\begin{equation}\label{eq:sampling}
 \sup_{|\xi|\leq N^3}
 \left|\frac1N\sum_{j=1}^N e^{-2\pi i\xi x_j}
          -\int_1^2 e^{-2\pi i\xi x}\,dx\right|
 \leq N^{-1/4}.
\end{equation}
\end{lem}

\begin{proof}
Let $X_1,\ldots,X_N$ be independent random variables, with
$X_j$ uniformly distributed on $I_j$. Then
$$
 \mathbb E\left(\frac1N\sum_{j=1}^N e^{-2\pi i\xi X_j}\right)
 =\sum_{j=1}^N\int_{I_j}e^{-2\pi i\xi x}\,dx
 =\int_1^2e^{-2\pi i\xi x}\,dx.
$$
Define
$$
 D_N(\xi)
 :=\frac1N\sum_{j=1}^N e^{-2\pi i\xi X_j}
       -\int_1^2e^{-2\pi i\xi x}\,dx.
$$
We first prove that, for every fixed $\xi\in\R$,
\begin{equation}\label{eq:concentration}
 \mathbb P\{|D_N(\xi)|>t\}\leq4e^{-Nt^2/4},
 \qquad t>0.
\end{equation}

Let $Y$ be a real random variable with values in $[-1,1]$,
and set $H(s):=\log\mathbb E e^{sY}$. A direct computation yields
$$
 H''(s)
 =\frac{\mathbb E(Y^2e^{sY})}{\mathbb E e^{sY}}
  -\left(\frac{\mathbb E(Ye^{sY})}{\mathbb E e^{sY}}\right)^2
 \leq1.
$$
Since $H(0)=0$ and $H'(0)=\mathbb E Y$, Taylor's formula gives
$H(s)\leq s\mathbb E Y+s^2/2$. Hence
$$
 \mathbb E e^{s(Y-\mathbb E Y)}\leq e^{s^2/2},
 \qquad s\in\R.
$$

Now let $Y_1,\ldots,Y_N$ be independent real random variables
with values in $[-1,1]$, and fix $s,u>0$.
By Markov's inequality, we have
$$
 \begin{aligned}
 &\mathbb P\left\{
   \frac1N\sum_{j=1}^N(Y_j-\mathbb E Y_j)>u
 \right\}\\
 &\quad=
 \mathbb P\left\{
   \exp\left(s\sum_{j=1}^N(Y_j-\mathbb E Y_j)\right)>e^{sNu}
 \right\}\\
 &\quad\leq e^{-sNu} \cdot 
   \mathbb E \left[ \exp\left(s\sum_{j=1}^N(Y_j-\mathbb E Y_j)\right) \right] \\
 &\quad= e^{-sNu}\prod_{j=1}^N
   \mathbb E e^{s(Y_j-\mathbb E Y_j)}\\
 &\quad\leq e^{-sNu+Ns^2/2}.
 \end{aligned}
$$
Letting $s=u$, we see the last expression equals $e^{-Nu^2/2}$.
The same argument applies to $-Y_j$, and therefore
\begin{equation} \label{eq:markov}
 \mathbb P\left\{
 \left|\frac1N\sum_{j=1}^N(Y_j-\mathbb E Y_j)\right|>u
 \right\}\leq2e^{-Nu^2/2}.
\end{equation} 

We now return to the proof of \eqref{eq:concentration}. For each fixed $\xi$, an easy computation yields
$$
 \re D_N(\xi)
 =\frac1N\sum_{j=1}^N
 \left(\cos(2\pi\xi X_j)-\mathbb E\cos(2\pi\xi X_j)\right).
$$
The same formula holds for $\im D_N(\xi)$ with
$\cos(2\pi\xi X_j)$ replaced by $-\sin(2\pi\xi X_j)$.
Each of these two families consists of independent random
variables with values in $[-1,1]$. Thus \eqref{eq:markov}
applies to both $\re D_N(\xi)$ and $\im D_N(\xi)$.
Moreover, if $|D_N(\xi)|>t$, then either
$|\re D_N(\xi)|>t/\sqrt2$ or
$|\im D_N(\xi)|>t/\sqrt2$. Consequently,
$$
 \begin{aligned}
 \mathbb P\{|D_N(\xi)|>t\}
 &\leq
 \mathbb P\{|\re D_N(\xi)|>t/\sqrt2\}
 +\mathbb P\{|\im D_N(\xi)|>t/\sqrt2\}\\
 &\leq4e^{-Nt^2/4}.
 \end{aligned}
$$
This proves \eqref{eq:concentration}.

\vspace{0.1cm}

We are now ready to prove \eqref{eq:sampling}. Since $1\leq X_j\leq2$,
we have
\begin{equation}\label{eq:derivative}
 |D_N'(\xi)|
 \leq\frac{2\pi}{N}\sum_{j=1}^N X_j
       +2\pi\int_1^2x\,dx
 \leq7\pi,
\end{equation}
where the bound is independent of $N$ and the
choice of $X_1,\ldots,X_N$.

Consider the finite set
$$
 \Gamma_N:=N^{-1}\Z\cap[-N^3,N^3].
$$
Then $\#\Gamma_N\lesssim N^4$, and every point of
$[-N^3,N^3]$ has distance at most $N^{-1}$ from $\Gamma_N$.
By \eqref{eq:concentration},
$$
 \begin{aligned}
 \mathbb P\left\{
   \max_{\xi\in\Gamma_N}|D_N(\xi)|>\frac12N^{-1/4}
 \right\}
 &\leq\sum_{\xi\in\Gamma_N}
   \mathbb P\left\{|D_N(\xi)|>\frac12N^{-1/4}\right\}\\
 &\lesssim \left(\#\Gamma_N \right) \cdot e^{-\sqrt N/16}\\
 &\lesssim N^4e^{-\sqrt N/16}.
 \end{aligned}
$$
The last expression tends to zero as $N\to\infty$.
We may therefore choose $N_0\geq2$ sufficiently large so that the probability
above is less than one for every $N\geq N_0$.
Hence, for each such $N$, there is a choice of $X_1,\ldots,X_N$
for which
$$
 |D_N(\xi)|\leq\frac12N^{-1/4}
 \qquad\text{for every }\xi\in\Gamma_N.
$$

Fix such a choice. For any $\xi\in[-N^3,N^3]$, choose
$\xi_0\in\Gamma_N$ with $|\xi-\xi_0|\leq N^{-1}$.
By \eqref{eq:derivative},
$$
 \begin{aligned}
 |D_N(\xi)|
 &\leq |D_N(\xi_0)|
       +\left|\int_{\xi_0}^{\xi}D_N'(s)\,ds\right|\\
 &\leq\frac12N^{-1/4}+7\pi|\xi-\xi_0|\\
 &\leq\frac12N^{-1/4}+7\pi N^{-1}\\
 &\leq N^{-1/4}.
 \end{aligned}
$$
Since $\xi\in[-N^3,N^3]$ was arbitrary, the points
$x_j:=X_j$, $1\leq j\leq N$, satisfy \eqref{eq:sampling}.
The proof is complete.
\end{proof}

\subsection{The Fourier decay estimate}

For each $N\geq N_0$, fix points $x_1,\ldots,x_N$ as in
Lemma~\ref{lem:sampling}. Define
\begin{equation}\label{eq:block}
 d\mu_N(x):=K_N(x)\,dx,\qquad
 K_N(x):=\frac N2\sum_{j=1}^N
   \one_{[x_j-N^{-2},\,x_j+N^{-2}]}(x).
\end{equation}
It is clear that $\mu_N$ is a positive and absolutely continuous measure with
$$
 \mu_N(\R)
 =\frac N2\sum_{j=1}^N 2N^{-2}=1.
$$
Thus $\mu_N$ is a probability measure.
Since $x_j\in[1,2]$ and $N^{-2}\leq1/4$, we also have
\begin{equation}\label{eq:common-support}
 \supp\;\mu_N\subseteq[1/2,5/2].
\end{equation}
We refer the reader to Figure~\ref{fig:mu-four} below for an illustration of this construction. 

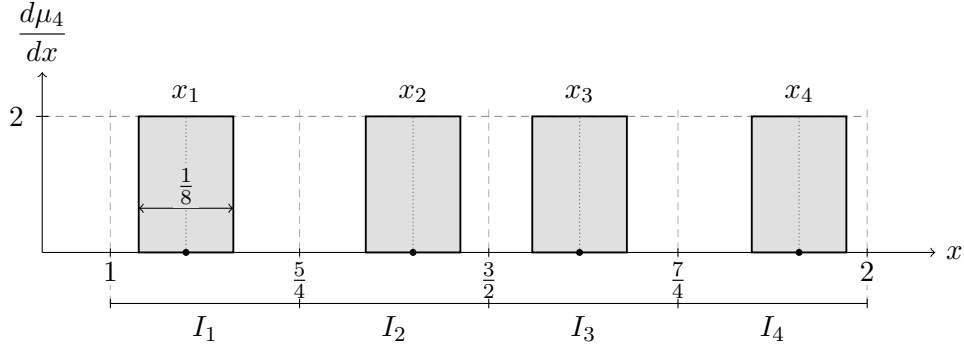
\begin{figure}[htbp]
\centering
\begin{tikzpicture}[x=10cm,y=0.9cm,font=\small]

% The four intervals I_j.
\foreach \p in {1,1.25,1.5,1.75,2}
  \draw[black!30,densely dashed] (\p,-0.75) -- (\p,2.15);
\draw[black!45,densely dashed] (0.91,2) -- (2,2);

% One illustrative choice of x_j in each I_j.
\foreach \p/\j in {1.10/1,1.40/2,1.62/3,1.91/4}{
  \draw[fill=black!12,line width=0.7pt]
    ({\p-0.0625},0) rectangle ({\p+0.0625},2);
  \draw[black!60,densely dotted] (\p,0) -- (\p,2);
  \fill (\p,0) circle (1.3pt);
  \node[above] at (\p,2.05) {$x_{\j}$};
}

% Axes and labels.
\draw[->] (0.91,0) -- (2.09,0) node[right] {$x$};
\draw[->] (0.91,0) -- (0.91,2.65)
  node[above] {$\dfrac{d\mu_4}{dx}$};
\draw (0.902,2) -- (0.918,2);
\node[left] at (0.90,2) {$2$};

\foreach \p/\v in {1/1,1.25/{\frac54},1.5/{\frac32},
                    1.75/{\frac74},2/2}{
  \draw (\p,0.06) -- (\p,-0.06);
  \node[below,fill=white,inner sep=1.5pt]
    at (\p,-0.08) {$\v$};
}

% Width of the first rectangle.
\draw[<->] (1.0375,0.65) -- (1.1625,0.65)
  node[midway,above,fill=black!12,inner sep=1.5pt]
  {$\frac18$};

% Labels for the intervals below the graph.
\draw (1,-0.75) -- (2,-0.75);
\foreach \p in {1,1.25,1.5,1.75,2}
  \draw (\p,-0.68) -- (\p,-0.82);
\foreach \p/\j in {1.125/1,1.375/2,1.625/3,1.875/4}
  \node[below] at (\p,-0.78) {$I_{\j}$};

\end{tikzpicture}
\caption{An illustration of the construction for $N=4$.
Each point $x_j$ lies in $I_j$, and each shaded rectangle has
width $1/8$, height $2$, and area $1/4$.}
\label{fig:mu-four}
\end{figure}

We have the following Fourier decay estimate for $\mu_N$. 

\begin{lem}\label{lem:uniform-decay}
For every $N\geq N_0$, the measure $\mu_N$ satisfies
\begin{equation}\label{eq:block-decay}
 |\widehat{\mu_N}(\xi)|\lesssim(1+|\xi|)^{-1/8},
 \qquad \xi\in\R,
\end{equation}
where the implicit constant is independent of $N$ and $\xi$.
\end{lem}

\begin{proof}
By \eqref{eq:block}, a direct computation yields
$$
 \begin{aligned}
 \widehat{\mu_N}(\xi)
 &=\frac N2\sum_{j=1}^N
   \int_{x_j-N^{-2}}^{x_j+N^{-2}}e^{-2\pi i\xi x}\,dx\\
 &=\left(\frac1N\sum_{j=1}^N e^{-2\pi i\xi x_j}\right)
   \frac{\sin(2\pi N^{-2}\xi)}{2\pi N^{-2}\xi}, \qquad \xi \neq 0,
 \end{aligned}
$$
and $\widehat{\mu_N}(0)=1$. Consequently,
\begin{equation}\label{eq:sinc}
 |\widehat{\mu_N}(\xi)|
 \leq\min\left\{1,\frac{N^2}{|\xi|}\right\}
      \left|\frac1N\sum_{j=1}^N e^{-2\pi i\xi x_j}\right|,
 \qquad \xi\neq0.
\end{equation}
Note that the absolute value of the average on the right hand side of \eqref{eq:sinc} is at most one, hence it suffices to show \eqref{eq:block-decay} for $|\xi| \ge 1$.

Since
$$
 \left|\int_1^2e^{-2\pi i\xi x}\,dx\right|
 \lesssim|\xi|^{-1},
 \qquad |\xi| \ge 1,
$$
By Lemma~\ref{lem:sampling}, we have 
\begin{equation}\label{eq:PN}
 \left|\frac1N\sum_{j=1}^N e^{-2\pi i\xi x_j}\right|
 \lesssim|\xi|^{-1}+N^{-1/4},
 \qquad 1\leq|\xi|\leq N^3.
\end{equation}
We now consider three different cases.

\vspace{0.1cm} 

\noindent $\bullet$ If $1\leq|\xi|\leq N^2$, then
$N^{-1/4}\leq|\xi|^{-1/8}$. By \eqref{eq:sinc}
and \eqref{eq:PN}, we have
$$
 |\widehat{\mu_N}(\xi)|
 \lesssim|\xi|^{-1}+N^{-1/4}
 \lesssim|\xi|^{-1/8}.
$$

\noindent $\bullet$ If $N^2\leq|\xi|\leq N^3$, then $N^{7/4}\leq|\xi|^{7/8}$.
Using again \eqref{eq:sinc}
and \eqref{eq:PN}, we have 
$$
 \begin{aligned}
 |\widehat{\mu_N}(\xi)|
 &\lesssim\frac{N^2}{|\xi|}
\left(|\xi|^{-1}+N^{-1/4}\right)=\frac{N^2}{|\xi|^2}+\frac{N^{7/4}}{|\xi|}\\
&\leq|\xi|^{-1}+|\xi|^{-1/8}\lesssim|\xi|^{-1/8}.
 \end{aligned}
$$

\noindent $\bullet$ Finally, if $|\xi|\geq N^3$, then by \eqref{eq:sinc}, we see that 
$$
 |\widehat{\mu_N}(\xi)|
 \leq\frac{N^2}{|\xi|}
 \leq|\xi|^{-1/3}
 \leq|\xi|^{-1/8}.
$$
Combining all three cases above, we deduce that
$$
 |\widehat{\mu_N}(\xi)|\lesssim|\xi|^{-1/8},
 \qquad |\xi|\geq1.
$$
The proof of \eqref{eq:block-decay} is complete. 
\end{proof}

\medskip 

\section{Proof of Theorem \ref{thm:main}}\label{sec:fixed}

\subsection{The lower bound for $M_{\mu_N}$}\label{sec:lower}

Recall that the points $x_1,\ldots,x_N$ satisfy
$$
 x_j\in I_j
 :=\left[1+\frac{j-1}{N},\,1+\frac jN\right],
 \qquad 1\leq j\leq N.
$$
We now prove that the best constant in the $H^1\to L^{1,\infty}$
estimate for $M_{\mu_N}$ tends to infinity along the sequence
$N=2^L$. The argument below only requires that $x_j\in I_j$.

\begin{prop}\label{prop:lower}
Let $N=2^L$, where $L \ge 1$ is a sufficiently large positive integer.
There exist an $H^1(\R)$ atom $a_N$ with $\|a_N\|_{L^1(\R)}\simeq1$
and pairwise disjoint measurable sets $E_k$, indexed by the integers
$L/4\leq k\leq L/2$, such that
\begin{equation}\label{eq:atom-lower}
 |E_k|\gtrsim1,\qquad
 a_N*(\mu_N)_k(x)\gtrsim1,\qquad x\in E_k.
\end{equation}
In particular,
\begin{equation}\label{eq:finite-ratio}
 \frac{\|M_{\mu_N}a_N\|_{L^{1,\infty}(\R)}}
      {\|a_N\|_{H^1(\R)}}
 \gtrsim L.
\end{equation}
All implicit constants are independent of $N$, $k$, and the
choice of the points $x_j\in I_j$.
\end{prop}

\begin{proof}
We divide the proof into four steps.

\vspace{0.1cm}

\noindent\textbf{Step 1: Construction of $a_N$.}
Define
\begin{equation} \label{eq:defnaN}
 a_N:=\frac N4\left(
       \one_{\left[-\frac1{10N},\,\frac1{10N}\right]}
       -\one_{\left[\frac9{10N},\,\frac{11}{10N}\right]}
       \right).
\end{equation}
Then $a_N$ is supported in $[-2/N,2/N]$, has integral zero,
and satisfies $\|a_N\|_{L^\infty(\R)}=N/4$.
Thus $a_N$ is an $H^1$ atom, and $\|a_N\|_{H^1(\R)}\leq1$.
A direct computation also gives $\|a_N\|_{L^1(\R)}\simeq1$.

\vspace{0.1cm}

\noindent\textbf{Step 2: Construction of $E_k$.}
Let
$$
 \mathcal J_N
 :=\left\{j\in\left\{\frac N4+1,\ldots,\frac{3N}4\right\}:
                  x_j-x_{j-1}\geq\frac1{2N}\right\}.
$$
We first show that $\#\mathcal J_N\simeq N$.
To see this, divide the indices in
$\{N/4+1,\ldots,3N/4\}$ into consecutive pairs $(j,j+1)$.
Since $x_{j+1}\in I_{j+1}$ and $x_{j-1}\in I_{j-1}$,
we have $x_{j+1}-x_{j-1}\geq 1/N$.
Hence at least one of the two gaps $x_j-x_{j-1}$ and
$x_{j+1}-x_j$ is at least $1/(2N)$.
Thus each pair $(j, j+1)$ contains at least an index in
$\mathcal J_N$, which proves the claim.

\vspace{0.1cm}

Note that $x_1\leq\cdots\leq x_N$. Hence, for
$j<j'$ with $j' \in \mathcal J_N$, we have
$$
 x_{j'}-x_j
 \geq x_{j'}-x_{j'-1}
 \geq\frac1{2N}.
$$
For each integer $L/4\leq k\leq L/2$, define
\begin{equation}\label{eq:Ek}
 E_k:=\bigcup_{j\in\mathcal J_N}
 \left(2^k x_j-\frac1{20N},\,
       2^k x_j+\frac1{20N}\right).
\end{equation}
The intervals in this union are pairwise disjoint, since their
centers are separated by at least $2^k/(2N)$, whereas each interval
has length $1/(10N)$. Consequently,
\begin{equation}\label{eq:Ek-size}
 |E_k|=\frac{\#\mathcal J_N}{10N}\simeq1.
\end{equation}
Moreover, for $j\in\mathcal J_N$, we have
$N/4+1\leq j\leq (3N)/4$ and $x_j\in I_j$, so
$x_j\in[5/4, 7/4]$, and consequently,
$$
 E_k
 \subseteq(2^k,2^{k+1}).
$$
Therefore the sets $E_k$ are also pairwise disjoint as $k$ varies.

\vspace{0.1cm}

\noindent\textbf{Step 3: Proof of \eqref{eq:atom-lower}.}
We now prove the pointwise lower bound in \eqref{eq:atom-lower}.
By \eqref{eq:block}, we have
$$
 a_N*(\mu_N)_k(x)
 =\frac N2\sum_{\ell=1}^N
   \int_{-N^{-2}}^{N^{-2}}
   a_N\bigl(x-2^k(x_\ell+u)\bigr)\,du.
$$
Fix $x\in E_k$ and choose $j(x) \in\mathcal J_N$ such that
$$
 |x-2^k x_{j(x)}|<\frac1{20N}.
$$
We split the sum above into three parts:
\begin{equation}\label{eq:convolution-split}
 \begin{aligned}
 a_N*(\mu_N)_k(x)
 &=\frac N2\int_{-N^{-2}}^{N^{-2}}
     a_N\bigl(x-2^k(x_{j(x)}+u)\bigr)\,du\\
 &\quad+\frac N2\sum_{\ell=1}^{j(x)-1}
     \int_{-N^{-2}}^{N^{-2}}
     a_N\bigl(x-2^k(x_\ell+u)\bigr)\,du\\
 &\quad+\frac N2\sum_{\ell=j(x)+1}^{N}
     \int_{-N^{-2}}^{N^{-2}}
     a_N\bigl(x-2^k(x_\ell+u)\bigr)\,du\\
 &=:A_1+A_2+A_3.
 \end{aligned}
\end{equation}

\vspace{0.1cm}

\noindent\textit{\underline{Estimate of $A_1$}}.
Since $k\leq\frac L2$ and $L$ is sufficiently large,
$$
 2^kN^{-2}\leq N^{-3/2}<\frac1{20N}.
$$
Hence, for every $|u|\leq N^{-2}$,
$$
 |x-2^k(x_{j(x)}+u)|
 \leq |x-2^k x_{j(x)}|+2^k|u|
 <\frac1{10N}.
$$
By \eqref{eq:defnaN}, this implies
$$
 a_N\bigl(x-2^k(x_{j(x)}+u)\bigr)=\frac N4,
 \qquad |u|\leq N^{-2}.
$$
Consequently,
$$
 A_1=\frac N2\cdot\frac2{N^2}\cdot\frac N4\simeq1.
$$

\vspace{0.1cm}

\noindent\textit{\underline{Estimate of $A_2$}}.
If $\ell<j(x)$, then $j(x)\in\mathcal J_N$ gives
$$
 x_{j(x)}-x_\ell
 \geq x_{j(x)}-x_{j(x)-1}
 \geq\frac1{2N}.
$$
Thus, for every $|u|\leq N^{-2}$,
$$
 \begin{aligned}
 x-2^k(x_\ell+u)
 &=x-2^k x_{j(x)}+2^k(x_{j(x)}-x_\ell)-2^k u\\
 &>\frac{2^k}{2N}-\frac1{10N}\\
 &>\frac2N,
 \end{aligned}
$$
where the last inequality follows from $k\geq L/4$.
Since $a_N$ is supported in
$\left[-2/N, 2/N\right]$, we have $A_2=0$.

\vspace{0.1cm}

\noindent\textit{\underline{Estimate of $A_3$}}. 
If $\ell>j(x)$, then $x_\ell\geq x_{j(x)}$, and 
$$
 x-2^k(x_\ell+u)
 \leq x-2^k x_{j(x)}+2^k|u|
 <\frac1{10N},
 \qquad |u|\leq N^{-2}.
$$
In particular, the above computation gives
$x-2^k(x_\ell+u) \notin
\left[\frac9{10N},\frac{11}{10N}\right]$,
where $a_N$ is negative, and hence $A_3\geq0$.

\vspace{0.1cm}

Substituting these estimates into \eqref{eq:convolution-split},
we obtain for $x \in E_k$,
\begin{equation}\label{eq:pointwise-lower}
 a_N*(\mu_N)_k(x)
 =A_1+A_2+A_3
 \geq A_1
 \gtrsim1.
\end{equation}
Together with \eqref{eq:Ek-size}, this proves
\eqref{eq:atom-lower}.

\vspace{0.1cm}

\noindent\textbf{Step 4: Proof of \eqref{eq:finite-ratio}.}
Finally, choose an absolute constant $\lambda>0$ smaller than
the uniform lower bound in \eqref{eq:pointwise-lower}. Then
$$
 \bigcup_{\substack{k\in\Z\\L/4\leq k\leq L/2}}E_k
 \subseteq\{M_{\mu_N}a_N>\lambda\}.
$$
Since the sets $E_k$ are pairwise disjoint and there are
$\simeq L$ integers in this range, we have
$$
 |\{M_{\mu_N}a_N>\lambda\}|
 \geq\sum_{\substack{k\in\Z\\L/4\leq k\leq L/2}}|E_k|
 \gtrsim L.
$$
Using $\|a_N\|_{H^1(\R)}\leq1$ and the fact that $\lambda$ is an
absolute constant, we obtain \eqref{eq:finite-ratio}.
The proof is complete.
\end{proof}

\subsection{Proof of Theorem~\ref{thm:main} for $d=1$}

Recall that $N_0$ is the integer introduced in
Lemma~\ref{lem:sampling}. We first construct the desired measure $\mu$. Fix a sufficiently large integer $q$
such that $2^{q^2}\geq N_0$, and define
\begin{equation}\label{eq:parameters}
 N_m:=2^{q^{2m}}\quad(m\geq1),\qquad
 K(x):=(q-1)\sum_{m=1}^\infty q^{-m}K_{N_m}(x),
\end{equation}
where we recall that $K_N(x)$ is defined as in \eqref{eq:block}.

Since each $K_{N_m}$ is nonnegative and has integral one, we have 
$$
 \int_{\mathbb R}K(x)\,dx
 =(q-1)\sum_{m=1}^\infty q^{-m}=1.
$$
Thus $d\mu(x):=K(x)\,dx$ defines an absolutely continuous
probability measure. By \eqref{eq:common-support}, its support
is contained in $[1/2,5/2]$. Moreover,
\begin{equation}\label{eq:mixture}
 \mu=(q-1)\sum_{m=1}^\infty q^{-m}\mu_{N_m},
\end{equation}
where $\mu_N$ is defined in \eqref{eq:block}. 
Therefore, by Lemma~\ref{lem:uniform-decay}, we have 
\begin{equation}\label{eq:fixed-decay}
 |\widehat\mu(\xi)|
 \leq(q-1)\sum_{m=1}^\infty
      q^{-m}|\widehat{\mu_{N_m}}(\xi)|
 \lesssim(1+|\xi|)^{-1/8},
\end{equation}
which verifies the Fourier decay condition \eqref{eq:main-decay}. 

\vspace{0.1cm}

We now prove that $M_\mu$ is not bounded from $H^1(\mathbb R)$
to $L^{1,\infty}(\mathbb R)$. Fix $m\geq1$ and apply Proposition~\ref{prop:lower}
with $N=N_m$, where we recall $N_m$ is defined as in \eqref{eq:parameters}. Let $a_{N_m}$ and $E_k$ be the resulting atom
and pairwise disjoint sets, where
\begin{equation} \label{eq:rangek}
 \frac14\log_2N_m\leq k\leq\frac12\log_2N_m.
\end{equation} 
We have $\|a_{N_m}\|_{H^1(\mathbb R)}\leq1$ and $|E_k|\gtrsim1$.
Moreover, there exists an absolute constant $c>0$, independent of
$q$, $m$, and $k$, such that
\begin{equation} \label{eq:lowerbound3.1}
 (q-1)q^{-m}a_{N_m}*(\mu_{N_m})_k(x)
 \geq 4c q^{1-m},
 \qquad x\in E_k.
\end{equation} 

By the definition of the weak $L^1$ quasi-norm,
\begin{equation}\label{eq:fixed-weak-start}
 \|M_\mu a_{N_m}\|_{L^{1,\infty}(\mathbb R)}
 \geq c q^{1-m}
 \left|\left\{x\in\mathbb R:
 M_\mu a_{N_m}(x)>c q^{1-m}\right\}\right|.
\end{equation}
We shall show that the level set on the right has measure
$\gtrsim\log_2N_m$.

For each integer $k$ in the range \eqref{eq:rangek}, \eqref{eq:mixture} gives
\begin{equation}\label{eq:fixed-split}
 \begin{aligned}
 a_{N_m}*\mu_k(x)
 &=(q-1)q^{-m}a_{N_m}*(\mu_{N_m})_k(x)\\
 &\quad+(q-1)\sum_{\ell<m}q^{-\ell}
       a_{N_m}*(\mu_{N_\ell})_k(x)\\
 &\quad+(q-1)\sum_{\ell>m}q^{-\ell}
       a_{N_m}*(\mu_{N_\ell})_k(x)\\
 &=:B_1(x)+B_2(x)+B_3(x).
 \end{aligned}
\end{equation}
We estimate these three terms separately.

\vspace{0.1cm}

\noindent\textit{\underline{Estimate of $B_1$}}.
By \eqref{eq:lowerbound3.1}, we have
$$
 B_1(x)\geq4c q^{1-m},
 \qquad x\in E_k.
$$

\vspace{0.1cm}

\noindent\textit{\underline{Estimate of $B_2$}}.
We may assume $m \ge 2$, otherwise $B_2=0$.
The Radon–Nikodym derivative with respect to Lebesgue measure of $(\mu_{N_\ell})_k$ is
$2^{-k}K_{N_\ell}(2^{-k}x)$, and \eqref{eq:block} gives
$\|K_{N_\ell}\|_{L^\infty(\mathbb R)}\lesssim N_\ell^2$.
Thus, by Young's inequality, we have 
\begin{equation}\label{eq:B2}
 \begin{aligned}
 \|B_2\|_{L^\infty(\mathbb R)}
 &\leq(q-1)\sum_{\ell<m}q^{-\ell}
       \|a_{N_m}\|_{L^1(\mathbb R)}
    \cdot 2^{-k}\|K_{N_\ell}\|_{L^\infty(\mathbb R)}\\
 &\lesssim 2^{-k}N_{m-1}^2 \leq N_m^{-1/4+2/q^2}\\
 &\leq N_m^{-1/8}
 \leq q^{-m}.
 \end{aligned}
\end{equation}
Here we used the facts that $\|a_{N_m}\|_{L^1(\mathbb R)}\lesssim1$,
$k\geq\frac14\log_2N_m$, and $N_{m-1}=N_m^{1/q^2}$.

\vspace{0.1cm}

\noindent\textit{\underline{Estimate of $B_3$}}.
Since each $(\mu_{N_\ell})_k$ is a probability measure, one has
$$
 \begin{aligned}
 \|B_3\|_{L^1(\mathbb R)}
 &\leq(q-1)\sum_{\ell>m}q^{-\ell}
       \|a_{N_m}*(\mu_{N_\ell})_k\|_{L^1(\mathbb R)}\\
 &\leq(q-1)\|a_{N_m}\|_{L^1(\mathbb R)}
       \sum_{\ell>m}q^{-\ell}\\
 &\lesssim q^{-m}.
 \end{aligned}
$$
Consequently, by Chebyshev's inequality, we have 
\begin{equation}\label{eq:B3}
 \left|\left\{x\in\mathbb R:
 |B_3(x)|>c q^{1-m}\right\}\right|
 \leq\frac{\|B_3\|_{L^1(\mathbb R)}}{c q^{1-m}}
 \lesssim q^{-1}.
\end{equation} 

\vspace{0.1cm}

Returning to the decomposition \eqref{eq:fixed-split}, define
$$
 E_{m,k}:=
 \left\{x\in E_k:|B_3(x)|\leq c q^{1-m}\right\}.
$$
Since $q$ is sufficiently large and $|E_k|\gtrsim1$,
\eqref{eq:B2} and \eqref{eq:B3} give
$$
 \|B_2\|_{L^\infty(\mathbb R)}\leq c q^{1-m},
 \qquad \textrm{and} \qquad 
 |E_{m,k}|\geq\frac12|E_k|\gtrsim1,
$$
uniformly in $m$ and $k$.

\vspace{0.1cm}

Therefore, when $x \in E_{m, k}$, we have 
\begin{equation}\label{eq:fixed-pointwise}
 \begin{aligned}
 a_{N_m}*\mu_k(x)
 &\geq B_1(x)-|B_2(x)|-|B_3(x)|\\
 &\geq4c q^{1-m}-c q^{1-m}-c q^{1-m}\\
 &=2c q^{1-m}.
 \end{aligned}
\end{equation}
Observe that
$$
 \bigcup_{\substack{k\in\mathbb Z\\
       \frac14\log_2N_m\leq k\leq\frac12\log_2N_m}}
 E_{m,k}
 \subseteq
 \left\{x\in\mathbb R:
 M_\mu a_{N_m}(x)>c q^{1-m}\right\}.
$$
Indeed, if $x\in E_{m,k}$ for some $k \in \left[ \frac14\log_2N_m, \frac12\log_2N_m \right] \cap \Z$, then by \eqref{eq:fixed-pointwise}, we have 
$$
 M_\mu a_{N_m}(x)
 =\sup_{j\in\mathbb Z}|a_{N_m}*\mu_j(x)|
 \geq |a_{N_m}*\mu_k(x)|
 \geq 2c q^{1-m}
 >c q^{1-m}, 
$$
which gives the desired observation.  Moreover, the sets $E_{m,k}$ are pairwise disjoint, since $E_{m,k}\subseteq E_k$, and each of them has measure $\gtrsim1$. Therefore,
$$
\left|\left\{x\in\mathbb R:
 M_\mu a_{N_m}(x)>c q^{1-m}\right\}\right| \ge 
 \sum_{\substack{k\in\mathbb Z\\
       \frac14\log_2N_m\leq k\leq\frac12\log_2N_m}}
 |E_{m,k}|
 \gtrsim\log_2N_m.
$$
Substituting this estimate into \eqref{eq:fixed-weak-start}, we obtain
\begin{equation}\label{eq:fixed-ratio}
 \|M_\mu a_{N_m}\|_{L^{1,\infty}(\mathbb R)}
 \gtrsim q^{1-m}\log_2N_m
 =q^{m+1}\longrightarrow\infty.
\end{equation}
Since $\|a_{N_m}\|_{H^1(\mathbb R)}\leq1$, this proves that
$M_\mu$ is not bounded from $H^1(\mathbb R)$ to
$L^{1,\infty}(\mathbb R)$. In particular, it is not of weak
type $(1,1)$, which proves Theorem~\ref{thm:main} for $d=1$.

\subsection{Proof of Theorem~\ref{thm:main} for $d\geq2$}

Next, we treat the higher dimensional case. Let $\mu$ be the measure constructed in \eqref{eq:mixture}, and recall that $d\mu(x_1)=K(x_1)\,dx_1$, where $K$
is defined in \eqref{eq:parameters}. Let $\eta$ be normalized
Lebesgue measure on $[-1,1]^{d-1}$, and define
$$
 \nu:=\mu\otimes\eta.
$$
Then $\nu$ is an absolutely continuous probability measure. More
precisely, writing $x=(x_1,x')\in\R\times\R^{d-1}$, we have
$$
 d\nu(x)
 =2^{-(d-1)}K(x_1)\one_{[-1,1]^{d-1}}(x')\,dx.
$$
Its support is contained in $[1/2,5/2]\times[-1,1]^{d-1}$.

We first verify the Fourier decay condition \eqref{eq:main-decay}. A direct computation yields
$$
 \widehat\nu(\xi)
 =\widehat\mu(\xi_1)
    \prod_{j=2}^d\frac{\sin(2\pi\xi_j)}{2\pi\xi_j}.
$$
Here, each quotient $\frac{\sin(2\pi\xi_j)}{2\pi\xi_j}$ is interpreted as $1$ when $\xi_j=0$.  Therefore,
by \eqref{eq:fixed-decay}, we have 
$$
 \begin{aligned}
 |\widehat\nu(\xi)|
 &\lesssim (1+|\xi_1|)^{-1/8}
       \prod_{j=2}^d(1+|\xi_j|)^{-1}\\
 &\leq \prod_{j=1}^d(1+|\xi_j|)^{-1/8}\\
 &\leq (1+|\xi|)^{-1/8}.
 \end{aligned}
$$
This verifies \eqref{eq:main-decay}.

\vspace{0.1cm}

Finally, we show that $M_\nu$ is not bounded from $H^1(\R^d)$
to $L^{1,\infty}(\R^d)$. For each $m\geq1$, define
$$
 F_m(x_1,x')
 :=a_{N_m}(x_1)\one_{[-N_m,N_m]^{d-1}}(x'),
$$
where $N_m$ and $a_{N_m}$ are defined as in \eqref{eq:parameters} and \eqref{eq:defnaN}, respectively. By the definition of the weak $L^1$ quasi-norm,
\begin{equation}\label{eq:product-weak-start}
 \|M_\nu F_m\|_{L^{1,\infty}(\R^d)}
 \geq c q^{1-m}
 \left|\left\{x\in\R^d:
 M_\nu F_m(x)>c q^{1-m}\right\}\right|,
\end{equation}
with the same constant $c$ as in \eqref{eq:fixed-pointwise}.

\vspace{0.1cm}

We first estimate $\|F_m\|_{H^1(\R^d)}$. Recall from
\eqref{eq:defnaN} that $a_{N_m}$ is supported in
$[-2/N_m,2/N_m]$, has integral zero, and satisfies
$\|a_{N_m}\|_{L^\infty(\R)}\leq N_m/4$.
Partition $[-N_m,N_m]^{d-1}$ into cubes $Q$ of side length
$4/N_m$, with disjoint interiors. Observe that for each such cube $Q$, the function
$$
a_{\widetilde{Q}}(x):=\frac{a_{N_m}(x_1)\one_Q(x')}{|Q|}
$$
has integral zero, is supported in the cube $\widetilde{Q}:=[-2/N_m,2/N_m]\times Q$, and has $L^\infty$ norm at most
$$
 \frac{N_m}{4|Q|}
 =\left(\frac4{N_m}\right)^{-d}=\left|\widetilde{Q}\right|^{-1}. 
$$
Therefore, $a_{\widetilde{Q}}$ is an $H^1(\R^d)$ atom. We therefore have the atomic decomposition
\begin{equation} \label{eq:testhigher}
 F_m(x)=\sum_Q |Q|\,a_{\widetilde{Q}}(x).
\end{equation}
Since $\|a_{\widetilde{Q}}\|_{H^1(\R^d)}\leq1$ for each $Q$,
we obtain
\begin{equation}\label{eq:product-hardy}
 \begin{aligned}
 \|F_m\|_{H^1(\R^d)}
 \leq\sum_Q |Q|\,\|a_{\widetilde{Q}}\|_{H^1(\R^d)}\leq\sum_Q |Q|
 \lesssim N_m^{d-1}.
 \end{aligned}
\end{equation}

\vspace{0.1cm}

We next estimate the measure of the level set in
\eqref{eq:product-weak-start}. Let $k$ be an integer in the
range \eqref{eq:rangek}. If
$x'\in[-N_m/2,N_m/2]^{d-1}$ and $y'\in[-1,1]^{d-1}$, then
$$
 |x'_j-2^k y'_j|
 \leq \frac{N_m}{2}+2^k
 \leq \frac{N_m}{2}+N_m^{1/2}
 \leq N_m
$$
for every coordinate $j$. Thus, 
$$
 \begin{aligned}
 F_m*\nu_k(x_1,x')
 &=\bigl(a_{N_m}*\mu_k(x_1)\bigr)
   \int_{\R^{d-1}}
   \one_{[-N_m,N_m]^{d-1}}(x'-2^k y')\,d\eta(y')\\
 &=a_{N_m}*\mu_k(x_1).
 \end{aligned}
$$
Consequently, if $x_1\in E_{m,k}$ and
$x'\in[-N_m/2,N_m/2]^{d-1}$, then
\eqref{eq:fixed-pointwise} gives
$$
 M_\nu F_m(x_1,x')
 \geq |F_m*\nu_k(x_1,x')|
 \geq 2c q^{1-m}
 >c q^{1-m}.
$$
Thus, it follows that
$$
 \left(
 \bigcup_{\substack{k\in\Z\\
       \frac14\log_2N_m\leq k\leq\frac12\log_2N_m}}
 E_{m,k}
 \right)
 \times[-N_m/2,N_m/2]^{d-1}
 \subseteq
 \left\{x\in\R^d:M_\nu F_m(x)>c q^{1-m}\right\}.
$$
Again, since the sets $E_{m,k}$ are pairwise disjoint and each has
measure $\gtrsim1$, we obtain
$$
 \begin{aligned}
 \left|\left\{x\in\R^d:
 M_\nu F_m(x)>c q^{1-m}\right\}\right|
 &\geq N_m^{d-1}
 \sum_{\substack{k\in\Z\\
       \frac14\log_2N_m\leq k\leq\frac12\log_2N_m}}
 |E_{m,k}|\\
 &\gtrsim N_m^{d-1}\log_2N_m.
 \end{aligned}
$$
Substituting this estimate into \eqref{eq:product-weak-start}
and using \eqref{eq:product-hardy}, we conclude that
$$
 \frac{\|M_\nu F_m\|_{L^{1,\infty}(\R^d)}}
      {\|F_m\|_{H^1(\R^d)}}
 \gtrsim q^{1-m}\log_2N_m
 =q^{m+1}\longrightarrow\infty.
$$
Thus $M_\nu$ is not bounded from $H^1(\R^d)$ to
$L^{1,\infty}(\R^d)$. In particular, it is not of weak type
$(1,1)$. This finishes the proof of Theorem~\ref{thm:main} for $d \ge 2$. 

\begin{rem}\label{rem:singular}
Wright also discusses the case of singular measures in connection
with Problems~14 and~15 in
\cite{GrafakosOliveiraPramanikSeegerStovall2017}.
The preceding construction also gives counterexamples in this class
for every $d\geq3$. Indeed, let $\mu$ be the one-dimensional
measure in \eqref{eq:mixture}, and let $\sigma$ now be normalized
surface measure on $\SSS^{d-2}\subset\R^{d-1}$. Set
$$
\widetilde{\nu}:=\mu\otimes \sigma.
$$
Then $\widetilde{\nu}$ is a compactly supported probability measure carried
by $[1/2,5/2]\times \SSS^{d-2}$, a set of zero $d$-dimensional
Lebesgue measure. Thus $\widetilde{\nu}$ is singular.

By \eqref{eq:fixed-decay} and the standard Fourier decay
estimate for surface measure on the sphere
(see \cite{Stein1993}), we have
$$
 |\widehat{\widetilde{\nu}}(\xi_1,\xi')|
 \lesssim
 (1+|\xi_1|)^{-1/8}
 (1+|\xi'|)^{-(d-2)/2}
 \lesssim
 (1+|(\xi_1,\xi')|)^{-1/8}.
$$
For the same test functions $F_m$ as in \eqref{eq:testhigher}, the atomic
decomposition proving \eqref{eq:product-hardy} remains valid.
Moreover, since $\sigma$ is a probability measure supported
in $[-1,1]^{d-1}$, we still have
$$
 F_m*\widetilde{\nu}_k(x_1,x')
 =a_{N_m}*\mu_k(x_1)
$$
whenever $k$ satisfies \eqref{eq:rangek} and
$x'\in[-N_m/2,N_m/2]^{d-1}$.
The same level-set argument therefore gives
$$
 \frac{\|M_{\widetilde{\nu}} F_m\|_{L^{1,\infty}(\R^d)}}
      {\|F_m\|_{H^1(\R^d)}}
 \gtrsim q^{1-m}\log_2N_m
 =q^{m+1}\longrightarrow\infty.
$$
Hence $M_{\widetilde{\nu}}$ fails to map $H^1(\R^d)$ boundedly into
$L^{1,\infty}(\R^d)$ and, in particular, is not of weak
type $(1,1)$. We leave the remaining details to the
interested reader.
\end{rem}

\bigskip 

\noindent {\bf Acknowledgments.} The author would like to thank
Bingyang Hu for suggesting this problem during his visit to
Shantou University in the summer of 2024 and for many helpful
discussions.

\medskip 

\noindent {\bf AI Disclosure.} The author used ChatGPT for language polishing and proofreading.
The idea of applying Markov's inequality to obtain \eqref{eq:markov}
in the proof of Lemma~\ref{lem:sampling} arose from discussions
with ChatGPT. The author takes full responsibility for the content
and conclusions of this paper.

\end{document}